\documentclass[11pt]{amsart}
\usepackage[letterpaper,margin=1.2in]{geometry}
\usepackage{amsbsy,amsfonts,amsmath,amssymb,amscd,amsthm,mathrsfs}
\usepackage{subfigure,enumitem,graphicx,epsfig,color}
\usepackage[colorlinks=true,linkcolor=blue,citecolor=green]{hyperref}
\normalsize

\newtheorem{theorem}{Theorem}[section]

\newtheorem{sublemma}[theorem]{Sublemma}

\newtheorem{definition}[theorem]{Definition}
\newtheorem{corollary}[theorem]{Corollary}
\newtheorem{remark}[theorem]{Remark}

\newtheorem{proposition}[theorem]{Proposition}

\newtheorem*{theorem A}{Theorem A}
\newtheorem*{corollary B}{Corollary B}
\newtheorem*{corollary C}{Corollary C}
\newtheorem*{theorem B}{Theorem B}
\theoremstyle{definition}
\newtheorem{example}[theorem]{Example}

\DeclareMathOperator{\IFS}{IFS}

\begin{document}
\title[Sensitivity of iterated function systems]{Sensitivity of iterated function systems}
\author[Ghane]{F. H. Ghane}
\address{\centerline{Department of Mathematics, Ferdowsi University of Mashhad}
\centerline{Mashhad, Iran.}  }
\email{ghane@math.um.ac.ir}
\email{f\_h\_ghane@yahoo.com}
\author[Rezaali]{E. Rezaali}
\address{\centerline{Department of Mathematics, Ferdowsi University of Mashhad}
\centerline{Mashhad, Iran.}}
\email{esmaeel.rezaali@gmail.com}
\author[Sarizadeh]{A. Sarizadeh$^{*}$}
\address{\centerline{Department of Mathematics, School of Science, Ilam University }
\centerline{P. O. Box: 69315-516, Ilam, Iran.}  }
\email{ali.sarizadeh@gmail.com}
\email{a.sarizadeh@ilam.ac.ir}
 \thanks{$^*$Corresponding author}
\begin{abstract}
The present work is concerned with the eqiucontinuity and sensitivity of iterated function systems (IFSs).
Here, we consider more general case of IFSs, i.e. the IFSs  generated by a family of relations.
We generalize the concepts of transitivity, sensitivity and equicontinuity to these kinds of systems.
This note investigates the relationships between these concepts.
Then, several sufficient conditions for sensitivity of IFSs are presented.
We introduce the notion of weak topologically exact for IFSs generated by a family of relations.
It is proved that non-minimal weak topologically exact IFSs are sensitive.
That yields to different examples of non-minimal sensitive systems which are not an $M$-system.
Moreover, some interesting examples are given which provide some facts about the sensitive property of IFSs.
\end{abstract}
\maketitle
\thispagestyle{empty}

\section{Introduction}
In deterministic dynamical systems, chaos concludes all random phenomena without any stochastic factors.
On the study of chaos, the concept of sensitivity  is a key ingredient.
In fact, sensitivity characterizes the unpredictability of
chaotic phenomena, and is one of the essential conditions in various definitions of chaos. Therefore, the
study on sensitivity has attracted a lot of attention from many researchers
(e.g. \cite{ABC, AG, ak,bbcds, GW,HS,km,ls}).
In 1971, Ruelle introduced the first precise definition for sensitivity \cite{RT}.
Then a formulation of sensitivity was given by Guckenheimer on the study of interval maps, \cite{G}.
In 1986, Devaney \cite{De} proposed the widely accepted definition of chaos (topological transitivity,
dense periodic points and sensitivity), and emphasized the significance of sensitivity in describing dynamical systems.
Afterwards, Li-Yorke sensitivity \cite{ak}, $n$-sensitivity \cite{Xe},
and collective sensitivity \cite{WW} were successively proposed, and each of
these concepts is used to describe the complexity of dynamical systems.
For continuous self-maps of compact metric spaces, Moothathu \cite{Mo}
initiated a preliminary study of stronger forms of sensitivity formulated in terms of large subsets of
$\mathbb{N}$. Mainly he considered syndetic sensitivity and cofinite sensitivity and proved that
any syndetically transitive, non-minimal map is syndetically sensitive.
Then, Liu, Liao and Wang \cite{LLW} introduced some other versions of sensitivities, including thick sensitivity and thickly syndetical
sensitivity.
After that Wang, Yin and Yan \cite{WYY} extended some of
these results to semigroup actions.
They presented some sufficient conditions for dynamical systems of
semigroup actions to have these sensitivities.

Although the sensitivity is widely understood as the central idea in Devaney chaos, but it
is implied by transitivity and density of periodic points \cite{bbcds}.
This result was generalized in \cite{AAB}, changing density of periodic points to density
of minimal points. In fact the authors proved that if a dynamical system $(X,f)$ is topologically
transitive and has no equicontinuity point, then it is sensitive. In particular a minimal
system is either equicontinuous or sensitive.
In \cite{GW}, Glasner and Weiss (see also Akin, Auslander and Berg \cite{AAB}) established a stronger result (for compact metric systems): they
proved that if $(X,f)$ is a non-minimal $M$-system then $(X,f)$
is sensitive.
For a compact metric space $X$, we recall that an $M$-system means that the set of almost
periodic points is dense in $X$ (the Bronstein condition) and, in addition,
the system is topologically transitive.
These results were generalized for groups
in \cite{G1}.
In \cite{km}, Kontorovich and Megrelishvili
generalized these results for a wide class of topological semigroup actions including one-parameter semigroup actions on Polish spaces and $M$-systems.
Notice that the dynamical characterizations of $M$-systems have received special attention, see e.g. \cite{G1, G2, HU}.

Recently, Iglesias and Portela \cite{IP} proved that if a semigroup is almost open then it is sensitive
or there exists a residual set of equicontinuity points. As a consequence they obtained
a sufficient condition for sensitivity that generalizes that given in Kontorovich and
Megrelishvili \cite{km}.

In this literature, Auslander-Yorke dichotomy theorem \cite{au} states that for a single transitive map, either
the system is sensitive, or the set of equicontinuity  points is exactly the same as the set of transitive points.
It holds in the non-invertible case as well.
On the other hand, any weak mixing system is sensitive and there are weak mixing systems which are not $M$-systems.
Akin \cite{A} provided an example of a mixing homeomorphism on the torus with a fixed point as the unique minimal set.
Furthermore, Akin and Kolyada in \cite{ak} introduced the notion of Li-Yorke sensitivity. They proved that every weak mixing system $(X, f)$, where $X$ is a compact metric space and $f$ a continuous map of $X$ is Li-Yorke sensitive and hence it is sensitive. An example of Li-Yorke sensitive system without weak mixing factors was given in \cite{Ci} (see also \cite{C2}).
In \cite{M1}, Mel$\acute{\text{i}}$chov$\acute{\text{a}}$ proved that every minimal system with a weak mixing factor, is Li-Yorke sensitive.

The objective of this paper is to discuss the sensitivity of iterated function systems.
An iterated function system, or IFS, is simply a finite collection of continuous self-maps of a topological space $X$.
Then we can consider the semigroup generated by these transformations.
IFSs provide a method for both generating and characterizing fractal images whenever the continuous maps are contraction.
Iterated function system was firstly introduced and then popularized by Hutchinson
\cite{H} and Barnsley \cite{B}. They have a wide variety of applications in many branches of nonlinear dynamics; for instance in the
image processing theory \cite{E} and in the theory of
stochastic growth models \cite{F}.
Here, we discuss a general case of IFSs, i.e. we consider the IFSs which are generated by a family of continuous relations.
We generalize some dynamical properties to continuous relations including  transitivity, equicontinuity and sensitivity.
Each IFS generated by a finite number of continuous relations $f_i:X \to X$, $i=1, \ldots, k$, can be considered as a continuous relation $F=\bigcup_{i=1}^k f_i$.
Then, the investigation of the relationship between these notions for IFSs can be done by using the relation $F$.
The concept of transitivity could trace back to
Birkhoff \cite{Br}. After that many articles dealt with such a topic.
Transitivity is a widely accepted feature of chaos. It is often required in definitions of chaos as one of several ingredients.
In the present work, we discuss some aspects of transitivity including transitivity, weak topological exactness and topological exactness
in iterated function systems (IFSs) and investigate their relations with the notion of sensitivity.
One of our main results states that non-minimal weak topologically exact IFSs are sensitive.
As is well known, unlike ordinary dynamical systems, the minimality of an iterated function system
$\IFS(X;\mathcal{F})$ with more than one generator
is not equivalent to that of its inverse $\IFS(X;\mathcal{F}^{-1})$.
To see such systems, we refer to \cite{BG}.
This property of IFSs yields to existence of different examples of non-minimal sensitive IFSs which are not $M$-system.
Additionally, we give an example of a sensitive minimal IFS that is not equicontinuous.
\subsection{The paper is organized as follows}
In Section 2, we recall some standard definitions about relations, continuous relations and iterated function systems generated by a family of relations.
We generalize some dynamical properties including transitivity, equicontinuity and sensitivity to continuous relations. Then, we investigate their relationships.
We introduce the notion of weak topologically exact property for IFSs.
It is proved that each IFS generated by a finite family of relations on a compact metric space so that their inverses are continuous relations is backward minimal iff it is weak topologically exact.
We study the sensitivity and equicontinuity properties for IFSs in  Section 3.
We show that every non-minimal topologically exact IFS generated by a finite family of maps on a compact metric space can not be equicontinuous.
Furthermore, we prove that each non-minimal topologically exact IFS generated by a finite family of open maps on a compact metric space is sensitive.
In Section 4, we give several examples of IFSs.
The first example shows that backward minimality can not be followed by minimality
but the second example is a non-minimal weak topologically exact IFS which is backward minimal.
This example provide a non-minimal sensitive systems which are not $M$-system.
The last example is a non-minimal topologically exact system which is not an $M$-system.
\section{Dynamical Systems of Relations}
A dynamical system in the present paper  is a triple $(\Gamma,X,\varphi)$, where $\Gamma$ is a
semigroup, $X$ is a set and
$\varphi : \Gamma\times X\to X$
is a relation. Sometimes we write the dynamical system as a pair $(\Gamma,X)$.
It can be said that is the most general and common definition of dynamical systems.

We remark that a \emph{relation} $f : X \to Y$ is a subset of $X \times Y$ with $f(x) =\{ y : (x, y) \in f\}$, for $x\in X$.
Following \cite{A}, the image of $A \subset X$ under  a relation $f : X \to Y$ is given by $f(A) =\bigcup_{x \in A}f(x)$.
We assume that $X$ and $Y$ are compact metric spaces.
We define $f^{-1}= \{(y,x): (x,y) \in f\}.$ Thus, $f$ is a map when $f(x)$ is a singleton for all $x \in X$.
We say that $f$ is \emph{surjective} when $f(X) = Y$ and $f^{-1}(Y ) = X$.
For $f : X \to Y$ and $g : Y \to Z$ the \emph{composition} $g \circ f : X \to Z$ is the
projection to $X \times Z$ of $(f \times Z) \cap (X \times g) \subset X \times Y \times Z$.
Inductively, for $f : X \to X$, $f^{n+1}: = f \circ f^n$ with $f^0:=id$ the identity map.

Let us now equip $X,Y$ with two topologies and assume that they are compact. A relation $f : X \to Y$ is \emph{closed} when it is a closed subset of $X \times Y$ or,
equivalently by compactness, when $f^{-1}(B)$ is closed for all closed set $B \subseteq Y$.
That is, $f$ is \emph{upper semi-continuous}.
 A relation $f : X \to Y$ is \emph{lower semi-continuous} when $f^{-1}(B)$ is open for all open set $B \subseteq Y$.
$f$ is said to be \emph{continuous} when it is closed and lower semi-continuous.
 That is $f$ is both upper and lower semi-continuous.
By compactness the composition of closed relations is closed.
A relation  $f$ is called \emph{bi-continuous} or \emph{open} if $f$ and $f^{-1}$ are continuous. The composition
of continuous (or bi-continuous) relations is continuous (resp. bi-continuous).
The finite union of closed, continuous or bi-continuous relations is satisfies
the corresponding property.

Let us remark two popular kinds of the dynamical systems: cascades and iterated function systems.
Consider the semigroup generated by a relation $f : X \to X$ and  denote it by $\Gamma$.
Take $\varphi : \Gamma\times X\to X$ so that $\varphi(f^n,x)=\{y\in X: (x,y)\in f^n\}$, for every $n\in \mathbb{N}$.
So, a pair $(\Gamma,X)$ is a dynamical system.
If $\Gamma = \{f^n\}_{n\in\mathbb{N}}$ and $f : X \to X$ is a relation,
then the classical dynamical system $(\Gamma,X)$ is called a \emph{cascade} and we use the standard
notation: $(X, f)$.

Now, let $\mathcal{F}$ be a family of relations defined on a set $X$.
We denote by
$\mathcal{F}^+$ the semigroup generated by these relations.
Take $\varphi : \Gamma\times X\to X$ so that $\varphi(h,x)=\{y\in X: (x,y)\in h\}$, where $\Gamma=\mathcal{F}^+$.
In this case, the dynamical system $(\mathcal{F}^+,X)$ is called an \emph{iterated function system} (IFS) associated to $\mathcal{F}$.
We use the usual
notations: $\IFS(X;\mathcal{F})$ or $\IFS(\mathcal{F})$.
Roughly speaking, an iterated function system (IFS) can be thought of as a collection of relations which can
be applied successively in any order.

For the $\IFS(X;\mathcal{F})$ and $x \in X$, the \emph{forward orbit} of
$x$ is defined by
$$\mathcal{O}_\mathcal{F}^+(x)= \bigcup_{h\in \mathcal{F}^+}h(x).$$
Since $f:X\to Y$ is a relation, so is $f^{-1}:Y\to X$. Analogously, one can define
the \emph{backward orbit} of $x$ by
$$\mathcal{O}_\mathcal{F}^-(x)= \bigcup_{h\in \mathcal{F}^+}h^{-1}(x),$$
where $h^{-1}(x)=\{y:(y,x)\in h\}$.
\begin{definition}\label{def1111}\cite{BG,km} Let $\IFS(\mathcal{F})$ be an iterated function system generated by a finite or infinite
family of relations, $\mathcal{F}$, on a metric space $X$.
\begin{enumerate}[label=(\roman*),ref=\roman*]
   \item\label{it:0} $\IFS(\mathcal{F})$ is called  \emph{symmetric},
                  if for each $f \in \mathcal{F}$ it holds that $f^{-1} \in \mathcal{F}$.
   \item\label{it:1} $\IFS(\mathcal{F})$ is called \emph{ transitive}, if for any two non-empty open sets $U$ and
                  $V$ in $X$, there exists $h \in \mathcal{F}^+$ such that $h(U) \cap V \neq \emptyset$.
   \item\label{it:2} A point $x$ is called a \emph{transitive point} if $\overline{\mathcal{O}_{\mathcal{F}}^+(x)}=X$,
                  where $\overline{A}$ is the closure of a subset $A$ of $X$. Let $\text{Trans}(\mathcal{F})$
                  denote the set of all transitive points.
   \item\label{it:3} $\IFS(\mathcal{F})$ is called  \emph{forward minimal}
                  if for every $x\in X$, $\text{Trans}(\mathcal{F})=X$.
   \item\label{it:4} $\IFS(\mathcal{F})$ is called
                  \emph{backward minimal}, if $\IFS(\mathcal{F}^{-1})$ is minimal where $\mathcal{F}^{-1}=\{f^{-1}: f\in \mathcal{F}\}$, i.e. $\overline{\mathcal{O}_{\mathcal{F}}^-(x)}=X$ for every $x\in X$.
   \item\label{it:4} A subset $A$ of $X$ is called \emph{forward invariant} for $\IFS(\mathcal{F})$
                  (or $\mathcal{F}$) if $f(A)\subset A$ for every $f\in \mathcal{F}$.
\end{enumerate}
\end{definition}
\begin{remark}\label{rem1-4}The following statements can be immediately deduced from the above discussions.
 \begin{itemize}
   \item\label{re:01} A subset $A$ is forward invariant for a relation $f:X\to X$ if and only if its complement is
                      forward invariant for $f^{-1}$.
   \item\label{re:02} A continuous map is a continuous relation.
   \item\label{re:04} If $\IFS(X,\mathcal{F})$ is transitive, so is $\IFS(X,\mathcal{F}^{-1})$, because for any two nonempty open sets $U$ and $V$, $f^n(U)\cap V \neq \emptyset$ iff $U \cap f^{-n}(V)\neq \emptyset$.
   \item\label{re:05} For an iterated function system $\IFS(X,\mathcal{F})$, if $\text{Trans}(\mathcal{F})$ is dense then $\IFS(\mathcal{F})$ is transitive.
                      When each $f$ in $\mathcal{F}$ is a continuous relation the converse is true.
   \item\label{re:06} If  $\IFS(X,\mathcal{F})$ is transitive then every non-empty open and $\mathcal{F}^{-1}$-invariant set $B$ is dense.
                     When each $f$ in $\mathcal{F}$ is a continuous relation the converse is true.
   \item\label{re:07} If $\IFS(X,\mathcal{F})$ is transitive and each $f$ in $\mathcal{F}$ is a continuous relation then $\text{Trans}(\mathcal{F})$ is a dense $G_\delta$ set whenever $X$ is compact metric.
   \item\label{re:08} The minimality of $\IFS(X;\mathcal{F})$ is equivalent to the following: every forward invariant non-empty closed subset
                      for $\IFS(\mathcal{F})$ is the whole space $X$.
 \end{itemize}
\end{remark}
In what follows, we introduce the concept of weak topological exactness  which has a key role in our results and is weaker than backward minimality.
\begin{definition}Let $\mathcal{F}$ be a finite or infinite family of relations on a topological space $X$.
\begin{enumerate}
  \item  We say that $\IFS(X;\mathcal{F})$ is weak topologically exact if for every open set $U$,  there exists a finite sequence
         $(T_i)_{i}^N$ in $\mathcal{F}^+$ so that $\bigcup_{i=1}^N T_i(U)$ is dense in $X$.
  \item  \cite{BFM2018}We say that $\IFS(X;\mathcal{F})$ is  topologically exact if for every open set $U$,  there exists a sequence
         $(T_i)_{i\in\mathbb{N}}$ in $\mathcal{F}^+$ so that $\bigcup_{i} T_i(U)=X$.
\end{enumerate}
\end{definition}
 The following proposition emphasizes that backward minimality is equivalent to weak topologically exactness, in the relations theme.
\begin{proposition}\label{stransi}
Let $\mathcal{F}$ be a family of relations on a compact normal space $X$ so that for every $f\in\mathcal{F}$, $f^{-1}$ is a continuous relation.
If $\IFS(\mathcal{F})$ is backward minimal and $U$ is
a nonempty open set then there exists $T_1,\dots,T_N\in \mathcal{F}^+$ such that
$\bigcup_{n=1}^N T_{n}(U)=X$.
Conversely, if for every nonempty open set $U$ there exists $T_1,\dots,T_N$ in $\mathcal{F}^+$ such that $\bigcup_{i=1}^{N} T_{i}(U)$
is dense in $X$ then $\IFS(\mathcal{F})$ is backward minimal.
\end{proposition}
\begin{proof}
Assume that $\IFS(\mathcal{F})$ is backward minimal and $U$ is a nonempty open set.
Since $\IFS(\mathcal{F}^{-1})$ is minimal, hence for each $x \in X$ there exist
$T \in \mathcal{F}^+$ and $y \in T^{-1}(x)$ so that $y \in U$ and therefore $x \in T(U)$.
By this fact and since for every $T\in \mathcal{F}^+$, the relation $T^{-1}$ is continuous and hence $T$ is open, so
$\{T(U) : T \in \mathcal{F}^+\}$ is an open cover and therefore, by compactness of $X$, one has a finite sub-cover, i.e.
there exist $T_i \in \mathcal{F}^+$, $i=1, \ldots, \ell$, such that $X=\bigcup_{i=1}^\ell T_i(U)$.

Conversely, assume $\IFS(\mathcal{F})$ is not backward minimal. Then there exists a closed, non-empty
proper subset $A$ of $X$ so that $A$ is invariant for $\IFS(X;\mathcal{F}^{-1})$.
Let $U$ be a non-empty open set whose closure $B$ is disjoint from $A$.
Since for every $f\in\mathcal{F}$, $f^{-1}$ is a continuous relation,
the set $\bigcup_{i=1}^\ell T_i(U)$ is an open subset of $X$ and a subset of the closed set
$\bigcup_{i=1}^\ell T_i(B)$, for every finite sequence $(T_i )_{i=1}^\ell$ in $\mathcal{F}^+$.
On the other hand, $\bigcup_{i=1}^\ell T_i(B)$ is disjoint from $A$ and so $\bigcup_{i=1}^\ell T_i(U)$ is not dense.
\end{proof}
\begin{remark}
Notice that the above argument works for IFSs generated by a family of relations.
For the IFSs generated that are generated by maps, the arguments and results may be slightly different.
For example, consider the expanding map $f:\mathbb{S}^1\to \mathbb{S}^1;\ x\mapsto  2x\ (\text{mod} 1)$.
This map is topologically exact and weak topologically exact but the relation $f^{-1}:\mathbb{S}^1\to \mathbb{S}^1$
is not a map and so we can not say about the backward minimality of $(\mathbb{S}^1,f)$ as a map.
Thus, it can not be considered as a counterexample for Proposition \ref{stransi}, in some sense.
\end{remark}
\begin{proposition}
Let $\mathcal{F}$ be a family of relations on a Lindel\"{o}f space $X$ so that for every $f\in\mathcal{F}$, $f^{-1}$ is a continuous relation.
If $\IFS(\mathcal{F})$ is backward minimal then it is topologically exact.
\end{proposition}
\begin{proof}
For given a nonempty open set $U$, let $A= X \setminus \bigcup_{T \in \mathcal{F}^+}T(U)$.
By Remark \ref{rem1-4}, $A$ is a backward invariant closed subset of $X$. Also, it is clear  $A \neq X$ (since $U \neq \emptyset$).
If the iterated function system $\IFS(\mathcal{F})$ is backward minimal,
then we have that $A=\emptyset$ and so, $X=\bigcup_{T \in \mathcal{F}^+}T(U) $.
Since $X$ is a Lindel\"{o}f space, there exists a sequence $(T_i)_{i\in\mathbb{N}}$ in
$\mathcal{F}$ such that $\bigcup_{i\in\mathbb{N}} T_{i}(U)=X$ and the proof is completed.
\end{proof}

In Section \ref{exam}, we give an example of non-minimal weak topologically exact IFS generated by maps which is backward minimal.
This example is important from another point of view.
In fact, it provides a non-minimal sensitive system which is not an $M$-system, see \cite{km}.

 \section{Equicontinuity and Sensitivity for iterated function systems}
In this section, we give several sufficient conditions for sensitivity of IFSs.  To this end,
we introduce the concept of sensitivity for IFSs which is a generalized version of the existing definition
for cascades.

Let $\mathcal{F}=\{f_i: X \to X:\ i=1,\dots,k\}$ be a finite family of relations defined on a compact metric space $(X,d)$.
Symbolic dynamic is a way to represent the elements of $\mathcal{F}^+$.
Indeed, consider the product space $\Sigma^+_k = \{1,\dots, k \}^\mathbb{N}$.
For any sequence $\omega=(\omega_1\omega_2\dots\omega_n\dots)\in \Sigma_k^+$,
 take $f^{0}_{\omega}:=Id$ and
$$
          f^n_{\omega}(x)=f^{n}_{\omega_{1}\dots\omega_{n}}(x)=f_{\omega_n}\circ f_\omega^{n-1}(x); \
          \ \forall \ n\in \mathbb{N}.
$$
Obviously,   $f^n_{\omega}=f^{n}_{\omega_{1}\dots\omega_{n}}=f_{\omega_n}\circ f_\omega^{n-1}\in \mathcal{F}^+$,
for every $ n\in \mathbb{N}$.

Define the metric
\begin{equation}\label{metric}
d_{\mathcal{F}}(x_1,x_2):=\sup_{\omega,n}d_H(f_\omega^n(x_1),f_\omega^n(x_2)),
\end{equation}
where $d_H$ is  the Hausdorff distance, or Hausdorff metric. Clearly,  $d_{\mathcal{F}} \geq d_H$.
Notice that if the generators are maps, then $d_H=d$.
\begin{definition} Let $\mathcal{F}$ be a finite family of relations on a compact metric space $X$.
\begin{enumerate}
  \item   A point $x$ is an equicontinuity point when the identity map from $(X; d)$ to $(X; d_{\mathcal{F}})$ is continuous at $x$.
          We denote the set of all equicontinuity points of $\IFS(\mathcal{F})$ by $\text{Eq}(X)$.
  \item   An $\IFS(\mathcal{F})$ is equicontinuous when every
          point is an equicontinuity point and so $d_{\mathcal{F}}$ is a metric equivalent to $d$ on $X$.
  \item   For $\varepsilon > 0$, define $\text{Eq}_\varepsilon$ to be the union of
  the open subsets with $d_{\mathcal{F}}$ diameter less than $\varepsilon$.
  \item   A point $x \in X$ is sensitivity point if it is not an equicontinuity point.
  \item   An $\IFS(\mathcal{F})$ is sensitive when there exists $\varepsilon > 0$ such that every nonempty open subset has
          $d_{\mathcal{F}}$ diameter at least $\varepsilon$, i.e. $\text{Eq}_\varepsilon=\emptyset$.
\end{enumerate}
 \end{definition}
\begin{remark}\label{rem21} Notice that
\begin{itemize}
  \item an $\IFS(\mathcal{F})$ is equicontinuous  if and only if $\text{Eq}(X)=X$;
  \item $Eq(X)=\bigcap_{\varepsilon>0}Eq_\varepsilon$.
\end{itemize}
\end{remark}
Hereafter, let $\mathcal{F}=\{f_i: X \to X:\ i=1,\dots,k\}$ be a finite family of continuous maps defined on a metric space $(X,d)$,
unless otherwise stated.
Let $X=\bigcup_{i=1}^kf_i(X)$. Consider the multifunction
$$
F:X\to \mathcal{P}(X);\ x\mapsto\bigcup_{i=1}^kf_i(x),$$
where $\mathcal{P}(X)$ is the set of all nonempty subsets of $X$.
The image of a nonempty
$B\in \mathcal{P}(X)$ under $F$ is $F(B) :=\bigcup_{b\in B} F(b)$.
By this, we can consider $(X,F)$ as a dynamical system.
Obviously, $F$ is a continuous relation.
Also, $F$ is onto which means that  for every $y\in X$ there exists $x\in X$ so that $y\in F(x)$ (or $\bigcup_{x\in X} F(x)=X$).

Notice that some topological properties of $(X,F)$ and $\IFS(\mathcal{F})$ are the same. For instance:
\begin{itemize}
  \item $A$ is an invariant set for  $\IFS(\mathcal{F})$ if and only if it is an invariant set for $F$;
  \item $(X,F)$ is transitive (resp. minimal) if and only if $\IFS(\mathcal{F})$ is transitive (resp. minimal);
  \item $\text{Trans}(\mathcal{F})$=$\text{Trans}(F)$.
\end{itemize}
So, by similar arguments to  \cite{AAB}, one can obtain some of its results for a cascade system $(X,F)$ and consequently, they hold for $\IFS(X;\mathcal{F})$.
\begin{proposition}\label{openmap}
Let $\mathcal{F}$ be a finite family of  continuous maps defined on a metric space $(X,d)$ with
$X=\bigcup_{f\in\mathcal{F}}f(X)$ and $F$ be as mentioned above.
\begin{enumerate}
  \item  If $F$ is an open relation then for every $\varepsilon>0$, $\text{Eq}_\varepsilon(X) $ is open and $F$-invariant.
  \item  If $(X,F)$ is transitive with open relation $F$ so that $Eq(X)\neq \emptyset$ then $Eq(X)$ is a dense $G_\delta$ set whenever $X$ is a complete metric space.
  \item  If $(X,F)$ is transitive then $\text{Eq}(X)\subseteq\text{Trans}(F)$.
  \item  If $F$ is an open relation and $\text{Eq}(X)\neq\emptyset$ then $\text{Trans}(F^{-1})\subseteq \text{Eq}(X)$.
\end{enumerate}
\end{proposition}
\begin{proof}
\begin{enumerate}
  \item That is an immediate consequence of the definition.
  \item  Indeed, $Eq(X)=\bigcap_{k\in \mathbb{N}}Eq_{1/k}$ and  $(X,F^{-1})$ is transitive.
         Since $Eq_{1/k}(X)$ is open and  $(F^{-1})^{-1}$-invariant, the set $Eq_{1/k}$ must be open and dense.
  \item Suppose that $x\in \text{Eq}(X)$ and $y$ is an arbitrary point of $X$. For every $\varepsilon>0$, there exists $\delta<\varepsilon$
         so that $d_\mathcal{F}$-diameter $U=B_d(x,\delta)$ less than $\varepsilon$.
         Since $(X,F)$ is transitive, there exists $T\in \mathcal{F}^+$ so that $T(U)\bigcap B(y,\varepsilon/2)\neq\emptyset$.
         Therefore, $\mathcal{O}^+_\mathcal{F}(x)\bigcap B(y,\varepsilon)\neq\emptyset$.
 \item   Let  $x$ be an arbitrary point in $\text{Trans}(F^{-1})$  and $\varepsilon>0$ be given.
         Since $\text{Eq}_\varepsilon(X)$ is open and $x\in \text{Trans}(F^{-1})$, one can have $F^{-n}(x)\bigcap\text{Eq}_\varepsilon(X)\neq\emptyset$.
         Take $y\in F^{-n}(x)\bigcap\text{Eq}_\varepsilon(X)$.
         When $\text{Eq}_\varepsilon(X)$ is $F$-invariant, $x\in F^{n}(y)\bigcap\text{Eq}_\varepsilon(X)$ and so $x\in \text{Eq}_\varepsilon(X)$, for
         every $\varepsilon>0$.
\end{enumerate}
\end{proof}
Summing up what in the previous corollary, we get the next result.
\begin{corollary}\label{corcor}
If  $(X,F)$ transitive with open relation $F$ and $\text{Eq}(X)\neq\emptyset$ then
$$
\text{Trans}(F^{-1})\subseteq \text{Eq}(X)\subseteq\text{Trans}(F).
$$
\end{corollary}
\begin{proposition}\label{review}
Suppose each $f_i:X\to X$ is a continuous open surjective map, for $i=1,\dots,k$, and $(X,F)$ is not sensitive.
If $(X,F^{-1})$ is minimal, then each $f_i$ is a homeomorphism and  $(X,F)$ is equicontinuous and minimal.
\end{proposition}
\begin{proof}
Since $(X,F)$ is not sensitive, $\text{Eq}_\varepsilon(X)\neq \emptyset$, for all $\varepsilon>0$.
By Proposition  \ref{openmap}, each $\text{Eq}_\varepsilon(X)$ is $F$-invariant.
So, minimality of $(X,F^{-1})$  and Corollary \ref{corcor} imply that $\text{Eq}_\varepsilon(X)=X$, for every $\varepsilon>0$.
Hence, $(X,F)$ is equicontinuous.

To prove the minimality of $(X,F)$, we redundant one of our assumptions.
In fact, the openness of $F$ is not a necessary condition in the following results.
Minimality can be followed from the following result.
\begin{sublemma}\label{22}
If $(X,F)$ is transitive and equicontinuous then it is minimal.
\end{sublemma}
\begin{proof}
Suppose that $A$ is a nonempty closed $F$-invariant subset of $X$.
Clearly, for every $i=1,\dots,k$, $f_i(A)\subset A$.
For arbitrary $\varepsilon>0$, the set $A_\varepsilon=\{x : d_{\mathcal{F}} (x,A) < \varepsilon\}$ is open and  $F$-invariant.
By Remark \ref{rem1-4}, $F^{-1}$ is transitive.
Thus, for arbitrary $\varepsilon>0$, $A_\varepsilon$ is dense and hence $\overline{A_\varepsilon}=X$, for arbitrary $\varepsilon>0$.
It follow that $A=X$ and so $F$ is minimal.
\end{proof}
At the end, notice that each $f_i$ is equicontinuous when  $(X,F)$ is equicontinuous.
The so-called Hippopotamus Hide Theorem \cite{A0} implies that each $f_i$ is an isometry and so is a homeomorphism.
\end{proof}
\begin{corollary}
Under the assumption of Proposition \ref{review}, $(X,F^{-1})$ is equicontinuous.
\end{corollary}

\begin{corollary}
Let $\mathcal{F}$ be a finite family of maps on a space $X$.
If $\IFS(\mathcal{F})$ is weak topologically exact but not minimal then $(X,F)$ can not be equicontinuous.
\end{corollary}
The following theorem says more than the mentioned corollary and
it insures that weak topologically exact systems with non-minimality assumption are sensitive.
Also, in the following theorem, we do not need where our generators are continuous.
\begin{theorem}\label{nonmi}
Let $\mathcal{F}=\{f_1,\dots,f_k\}$ be a finite family of maps on a
compact space $X$ so that  $f^{-1}_i$ is continuous relations, $i=1,\dots,k$.
If $\IFS(\mathcal{F})$ is weak topologically exact but not minimal then $\IFS(\mathcal{F})$ is sensitive.
\end{theorem}
\begin{proof}
Let $y\in X$ with  $\overline{\mathcal{O}^+_{\mathcal{F}}(y)}\neq X$ and $z\in X \setminus \overline{\mathcal{O}^+_{\mathcal{F}}(y)}$.
Take $V=B(z,\delta)$ where $\delta=\frac{1}{4}d(z,\overline{\mathcal{O}^+_{\mathcal{F}}(y)}).$
Let $x \in X$, and let $U\subset X$ be an arbitrary neighborhood of $x$.
Since $\IFS(\mathcal{F})$ is weak topological exact, there exist  $T_1,\dots,T_\ell$ in $\mathcal{F}^+$ so that
the following holds:

$(H_1) \ \  X\subseteq \overline{\bigcup_{i=1}^\ell T_i(U)}.$\\
Furthermore, for every $i\in \{1,\dots,\ell\}$, there exist $T^{(i)}_1,\dots,T^{(i)}_{\ell_i}$ so that

 $(H_2) \ \  X\subseteq \overline{\bigcup_{j=1}^{\ell_i} T^{(i)}_j(T_i(U))}; \ \forall \ 1\leq i\leq\ell.$\\
 To verify the condition $(H_2)$ notice that $T_i^{-1}$ is a continuous relation,
 hence $T_i(U)$ is a nonempty open subset of $X$. Therefore, $(H_2)$ follows by $(H_1)$.

Take $t:= \ell+\max \Xi$, where $\Xi=\{|T^{(i)}_j|: 1\leq i\leq\ell\ \& \ 1\leq j\leq \ell_i\}$
and $|T|$ is the length of $T$ (we say that the length of $T$ is equal to $n$ and write $|T|=n$ if $T$ is the combination of $n$ elements of the generating set $\mathcal{F}$). Choose a neighborhood  $W$ around $y$ such that
$\text{diam}(f_\rho^i(W))<\delta$, for every $\rho\in \Sigma^+_k$ and for every $i=0, 1, \dots, t$.
Moreover, by the choice of $\delta$ we may assume that $d(f_\rho^i(W),V)\geq 2\delta$, for $i=0, 1, \ldots,t$ and for each $\rho\in\Sigma^+_k$.

The condition $(H_1)$ ensures that for some $1\leq s\leq \ell$,
$T_{s}(U)\bigcap W\neq \emptyset$.
On the other hand, according to $(H_2)$, there exists $1\leq j_0\leq \ell_s$ so that
$T^{(s)}_{j_0}(T_s(U))\bigcap V\neq \emptyset$.
Also, one has that $T^{(s)}_{j_0}(T_s(U))\bigcap T^{(s)}_{j_0}(W)\neq \emptyset$.
By the choice of $W$, we have $d(T^{(s)}_{j_0}(W),V)\geq 2\delta$.
Thus, $\text{diam}(T^{(s)}_{j_0}(T_s(U)))>\delta$.
Since $x$ and $U$ are arbitrary and $\delta$ does not depend on $x$, the proof is complete.
\end{proof}
Now, consider a Barnsley $\IFS(X;\mathcal{F})$ i.e. $f_i$ is a contraction of the metric $d$ on $X$,  for each $i=1,\dots,k$, where $(X,d)$ is a compact metric space.
We remark that $f$ is contractive if there exists $\lambda<1$ such that $d(f(x_1),f(x_2))\leq\lambda d(x_1,x_2)$.
Also, let $X=\bigcup_{i=1}^kf_i(X)$.
In this case, $d_\mathcal{F}=d$ and so $\IFS(\mathcal{F})$ is equicontinuous.
The induced map on the space of nonempty closed subsets
of $X$ equipped with Hausdorff metric is contraction and so by the Banach fixed point theorem it has a unique fixed point.
Since we have assumed that $F$ is surjective, the unique fixed point is $X$.
It follows that $(X,F)$ is minimal and so $(X, F^{-1})$ is transitive but it need not be minimal.
By Theorem \ref{nonmi},  $(X, F^{-1})$ is sensitive.
In this way, a natural question arises for $\IFS(\mathcal{F})$ generated by homeomorphisms where
$\text{Eq}(X)$ is a dense $G_\delta$-subset: under which assumptions $\IFS(X;\mathcal{F}^{-1})$ is sensitive?

We finalize this section with another sufficient condition for sensitivity.
Let $f$ be a map on a metric space $(X,d)$ and $x \in X$ be a fixed point of $f$.
Then the stable set $W^s(x)$ and unstable set $W^u(x)$ are defined, respectively, by
$$W^s(x)=\{y \in X: f^n(y)\to x, \ \textnormal{as} \ n \to +\infty\},$$
$$W^u(x)=\{y \in X: f^n(y)\to x, \ \textnormal{as} \ n \to -\infty\}.$$
\begin{definition}
We say that $x$ is an attracting (or repelling) fixed point of $f$ if the stable set $W^s(x)$ (or unstable set $W^u(x)$) contains an open neighborhood $B$ of $x$.
Then $B$ is called local basin of attraction (or repulsion) of $x$.
\end{definition}
\begin{theorem}\label{pro1}
Let $\mathcal{F}$ be a finite family of open continuous maps.
If $\IFS(\mathcal{F})$ is a weak topologically exact iterated function system and
the associated semigroup $\mathcal{F}^+$ has a map with a repelling
fixed point then $\IFS(\mathcal{F})$ is sensitive on $X$.
\end{theorem}
\begin{proof}
Suppose that the associated semigroup $\mathcal{F}^+$ has a map $h$ with a repelling fixed point $q$ and the local repulsion basin $B$.
Suppose that $\textnormal{IFS}(\mathcal{F})$ is not sensitive.
Then for each $n \in \mathbb{N}$, there is a non-empty open subset $U_n$ of $X$ such that the following holds:
\begin{equation}\label{e11}
diam(f_\omega^i(U_n))< \frac{1}{n}; \ \ \forall\omega \in \Sigma_k^+ \ \forall i \in \mathbb{N}.
\end{equation}
Take $n$ large enough, then
for each $y, z \in B$ with $y\neq z$, there exists $m \in \mathbb{N}$ sufficiently large so that
\begin{equation}\label{e12}
d(h^m(y),h^m(z)) \geq \frac{1}{n}.
\end{equation}
Let us fix the integer $n \in \mathbb{N}$ for which (\ref{e12}) holds. Since $\textnormal{IFS}(\mathcal{F})$ is weak topologically exact and hence by Proposition \ref{stransi} it is backward minimal, there is $T \in \mathcal{F}^+$ and $q^{\prime}\in T^{-1}(q)$ so that $q^{\prime} \in U_n$.
By continuity there is $\delta > 0$ such that $B_\delta(q) \subset B$ and there exists $B^{\prime} \subset T^{-1}(B_\delta(q))$, with $q^{\prime} \in B^{\prime}$ and $T(B^{\prime})=B_\delta(q)$, so that $B^{\prime}\subset U_n$ which implies that
$B_\delta(q) \subset T(U_n)$.\\
Let us take $y,z \in B_\delta(q)$, $y^{\prime}\in T^{-1}(y)\cap B^{\prime}$ and $z^{\prime}\in T^{-1}(z)\cap B^{\prime}$. Then $y^{\prime},z^{\prime} \in B^{\prime} \subset U_n$ and hence by (\ref{e12})
 $$d(h^m \circ T(y^{\prime}),h^m \circ T(z^{\prime})) = d(h^m(y),h^m(z)) \geq \frac{1}{n}$$
which contradicts (\ref{e11}).
\end{proof}
\section{Examples}\label{exam}
The first example belong to the folklore.
\begin{example}
Let us equipped the product space $\Sigma^+_2=\{1,2\}^\mathbb{N}$ to the metric
$d(\eta,\omega)=\inf\{2^{1-n}:\ \eta_i=\omega_i, \forall i<n\}$.
For $\theta=1,2$, let $(f_\theta(\eta))_1=\theta$ and $(f_\theta(\eta))_i=\eta_{i-1}$.
Take $\mathcal{F}=\{f_1,f_2\}$.
The inverse $F^{-1}$ on $\Sigma^+_2=\{1,2\}^\mathbb{N}$ is shift map given by $(F^{-1}(\eta))_i=(\sigma(\eta))=\eta_{i+1}$
Thus, the periodic points of $F^{-1}$ is dense and $F^{-1}$ is sensitive.
However, the $\IFS(\mathcal{F})$ is minimal and equicontinuous.
\end{example}
It is a well known fact that for ordinary dynamical
systems, the minimality of a map $f$ is equivalent to that of $f^{-1}$. Nevertheless this is
not the case for iterated function systems as Kleptsyn and Nalskii pointed at ~\cite[pg.~271]{KN04}.
 However, they omitted to include any example of forward but not backward
minimal IFS. Recently, the authors in \cite{BG}, illustrated an example of forward but not backward minimal IFS which we expose  below.
That lead to an example of backward minimal IFS which is not forward minimal. So, it is sensitive.
Since the previous examples of non-minimal sensitive systems (see, \cite{AAB,GW}) are $M$-system, this example maybe more valuable.
Indeed, our example is non-minimal sensitive system which is not an $M$-system.
Recall that $(S,X)$  is an \emph{$M$-system} if the set of almost periodic points\footnote{
A point $x$ is called almost periodic if the subsystem $\overline{\mathcal{O}^+_\mathcal{F}(x)}$ is minimal and compact.} is dense in $X$ (the
Bronstein condition) and, in addition, the system is transitive.
\begin{example}
Consider a symmetric $\IFS(\mathcal{F})$ generated by homeomorphisms of the circle.
Then, one of the following possibilities holds
~\cite{Navas,Gh01}:
\begin{enumerate}[itemsep=0.1cm]
\item there is a point $x$ so that $\text{Card}(\mathcal{O}^+_\mathcal{F}(x))<\infty$;
\item $\IFS(\mathcal{F})$ is minimal; or
\item there is a unique invariant minimal Cantor set $K$ for $\IFS(\mathcal{F})$, that is,
$$
g(K)=K \ \ \text{for all $g\in \mathcal{F}$} \quad  \text{and} \quad
K=\overline{\mathcal{O}^+_\mathcal{F}(x)} \ \
   \text{for all $x\in K$}.
$$.
\end{enumerate}
 The Cantor set $K$ in the above third conclusion is
usually called \emph{exceptional minimal set}.
By~\cite[Exer.~2.1.5]{Navas}, one can ensure that
there exists a symmetric semigroup $\mathcal{F}^+$ generated by homeomorphisms  $f_1,\dots,f_k$ of
$\mathbb{S}^1$ which admits an exceptional minimal set $K$
such that the orbit of every point of $\mathbb{S}^1\setminus K$ is dense in $\mathbb{S}^1$.
So, the
closed invariant subsets of $\mathbb{S}^1$  for $\IFS(\mathcal{F})$ are $\emptyset$,
$K$ and $\mathbb{S}^1$.

Now, consider any  homeomorphism $h$ of $\mathbb{S}^1$ such that $h(K)$
strictly contains $K$. Then the IFS generated by
$f_1^{-1},\ldots,f_n^{-1},h^{-1}$ is backward minimal but not forward minimal.
Therefore, the IFS generated by
$f_1^{-1},\ldots,f_n^{-1},h^{-1}$ is weak topological exact and sensitive. Also,
one can easily check that the IFS generated by
$f_1^{-1},\ldots,f_n^{-1},h^{-1}$  is not an $M$-system.
\end{example}
As a consequence of the previous example we get the next result.
\begin{corollary}
There exists a weak topologically exact non minimal IFS which is sensitive but it is not an $M$-system.
\end{corollary}
Finally, we provide another example of non-minimal topologically exact system which is not an $M$-system.
\begin{example}
Suppose $f$ is a north-south pole homeomorphism on the circle.
By this we mean that the nonwandering
set of $f$, $\Omega(f)$, consists of one fixed source, $q$, one fixed sink, $p$. 
Then 
$$\Omega(f)=\{p\}\cup\{q\}.$$
The salient feature of the north-south pole homeomorphism
is that for every $x\in \mathbb{S}^1\setminus\Omega(f)$, 
$f^n(x)\to p$ and $f^{-n}(x)\to q$ as $n\to+\infty$.
It is observed that $\mathbb{S}^1\setminus\{p,q\}$ composed of exactly two connected pieces $U_1$ and $U_2$.
Now, consider continuous maps $h$, $h_1$ and $h_2$ of $\mathbb{S}^1$ such that
\begin{enumerate}
    \item $h(x)=p$, for every $x\in \mathbb{S}^1$,
    \item $h_1(p)=p=h_2(p)$,
    \item there exist connected closed set $I$ of the circle so that $q\in I$ and $p$, $q$ are not boundary points for $I$,
    \item $\IFS(I,h_1|_I,h_2|_I)$ is minimal,
    \item there exist $a\in I\cap U_1$ and $b\in I\cap U_2$ so that $f(a),f(b)\in I$.
\end{enumerate}
Take $\mathcal{F}=\{f,f^{-1},h,h_1,h_2\}$. One can check
$$
        \overline{\mathcal{O}^+_\mathcal{F}(x)}=\mathbb{S}^1;\ \ \ \ \forall \  x\in \mathbb{S}^1\setminus \{p\}.
$$
Also, $g(p)=p$,  for every $g\in \mathcal{F}^+$.
Therefore, by Definition \ref{def1111}, the $\IFS(\mathbb{S}^1;\mathcal{F})$ is neither forward minimal nor backward minimal.
However, there are two important facts:
\begin{enumerate}
  \item  $\IFS(\mathbb{S}^1;\mathcal{F})$ is topologically exact,
  \item $\IFS(\mathbb{S}^1;\mathcal{F})$ is not an $M$-system.
\end{enumerate}
Indeed, $x$ is not almost periodic point when $x\neq p$ i.e. $\overline{\mathcal{O}^+_\mathcal{F}(x)}$ is not minimal  for $x\neq p$.
\end{example}

\section*{Acknowledgments}
We would like to thank the anonymous reviewer whose comments
and remarks improved the results and presentation of the paper.

\end{document}